\documentclass[11pt,reqno]{amsart}
\usepackage{tikz}
\usepackage{subfig}
\usepackage{epstopdf}
\usepackage{epsfig}
\usepackage{math}
\graphicspath{{./figures/}}
\usepackage{tikz}
\usetikzlibrary{shapes,arrows}
\tikzstyle{decision} = [diamond, draw, fill=blue!20, 
    text width=4.5em, text badly centered, node distance=3cm, inner sep=0pt]
\tikzstyle{block} = [rectangle, draw, fill=blue!20, 
    text width=5em, text centered, rounded corners, minimum height=4em]
\tikzstyle{line} = [draw, -latex']
\tikzstyle{cloud} = [draw, ellipse,fill=red!20, node distance=3cm,
    minimum height=2em]

\newcommand{\ts}{\mathsf{T}}

\usetikzlibrary{positioning}
\tikzset{main node/.style={circle,fill=blue!20,draw,minimum size=1cm,inner sep=0pt},  }

\begin{document}
\title[]{Entropy dissipation via Information Gamma calculus: Non-reversible stochastic differential equations}
\author[Feng]{Qi Feng}
\author[Li]{Wuchen Li}
\email{qif@usc.edu;}
\email{wuchen@mailbox.sc.edu;}
\address{Department of Mathematics, University of Southern California, Los Angeles, CA 90089;}
\address{Department of Mathematics, University of South Carolina,  Columbia, SC 29208.}

\thanks{Wuchen Li is supported by University of South Carolina, start--up funding}
\begin{abstract}
We formulate explicit bounds to guarantee the exponential dissipation for some non-gradient stochastic differential equations towards their invariant distributions. Our method extends the connection between Gamma calculus and Hessian operators in $L^2$--Wasserstein space. In details, we apply Lyapunov methods in the space of probabilities, where the Lyapunov functional is chosen as the relative Fisher information. We derive the Fisher information induced Gamma calculus to handle non-gradient drift vector fields. We obtain the explicit dissipation bound in terms of $L_1$ distance and formulate the non-reversible Poincar{\'e} inequality. An analytical example is provided for a non-reversible Langevin dynamic. 
\end{abstract}
\maketitle
\section{Introduction}
The convergence behaviors of non-reversible stochastic dynamical systems towards their invariant distributions play essential roles in probability \cite{AE, BE2008, BGL, CG2017,Guillin, Hwang1993, Hwang2005, KM2012, LNP2013, MO}, functional inequalities \cite{C1, Gross} and dynamical systems \cite{Legoll, Luc} with applications in Bayesian sampling problems \cite{non-reversible}. Consider a stochastic differential equation (SDE)
\begin{equation*} 
dX_t=b(X_t)dt+\sqrt{2} dB_t,
\end{equation*}
where $X_t\in \mathbb{R}^d$, $b\colon \mathbb{R}^d\rightarrow\mathbb{R}^d$ is a given smooth drift function and $B_t$ is a standard Brownian motion in $\mathbb{R}^d$. 

If the drift vector $b$ is a gradient vector of a given potential function $U$, i.e., $b=\nabla U$, the above SDE is known as the over-damped Langevin dynamics. It is reversible with many classical convergence properties \cite{BE, Villani2009_optimal}. A well-known fact is that there are several Lyapunov functionals such as relative entropy and relative Fisher information between the initial density and the associated invariant distribution. Under suitable conditions of $U$, one shows that these functionals decay exponentially fast, namely the entropy dissipation \cite{MV}. Mathematically, one way of obtaining these conditions is the Bakry--{\'E}mery iterative Gamma calculus \cite{BE}. This calculus can be understood as a direct calculation method to obtain dissipation rates. Nowadays, this Gamma calculus is connected with Hessian operators in $L^2$--Wasserstein space \cite{OV, Villani2009_optimal}. If the drift vector $b$ is not a gradient vector field, this SDE is non-reversible. Here \cite{AC, ACJ, FJ} have studied the entropy dissipation results with explicit rates for non-gradient stochastic dynamical systems. In particular, the Arnold-Carlen tensor \cite{AC} has been formulated to guarantee the exponential convergence of non-reversible SDEs, with extensions in trajectory and path space \cite{FJ}. 

This paper provides the other entropy dissipation results for non-reversible SDEs with an explicitly known invariant distribution. We also choose the relative Fisher information as the Lyapunov functional for these non-reversible SDEs. Following the Hessian operators of KL divergence in $L^2$--Wasserstein space, we formulate a new iterative Gamma calculus, which  addresses those non-gradient drift vector fields. And we derive sufficient conditions to guarantee the exponential convergence in term of $L_1$ distance with explicit bounds. Following the proposed Gamma calculus, we derive the Poincar{\'e} inequality for non-reversible SDEs. An example of an explicit rate is provided, for $b=-(\mathbb{I}+\mathbb{J})\nabla U$, where $\mathbb{J}$ is a given constant skew-symmetric matrix. 

In literature, several methods address convergence behaviors of non--reversible stochastic dynamical systems. A celebrated result is known as the Foster-Lyapunov criterion \cite{MT}. It requires to construct a Lyapunov function in a spatial domain from the SDE's generator function. They can prove the convergence in the sense of semi-group. And there are several other dissipation results \cite{AM2019,BGH2019}. In particular, \cite{AC, ACJ} formulate the entropy method and condition for nonsymmetry Fokker-Planck equations. They derive the Arnold-Carlen tensor for generalized Bakry--{\'E}mery condition, which can explicitly guarantee the convergence behavior globally in space. Compared to existing works, we provide the other explicit convergence bound for some non-reversible stochastic systems in the sense of relative entropy and relative Fisher information. Our results derive dissipation rates by introducing {\em information Gamma calculus}, which works for non-gradient Fokker-Planck equations. Besides, our derivation extends the connection between geometric calculations in $L^2$--Wasserstein space and Gamma calculus \cite{Li2018_geometrya,OV,Villani2009_optimal}, which goes beyond gradient drift vector fields. See related generalized Gamma calculus in \cite{FL, LiG2}.

This paper is organized as follows. In section \ref{sec2}, we present the main results of this paper. We formulate the explicit condition for the exponential decay of Fisher information. And we derive the convergence result in terms of relative entropy, $L^2$--Wasserstein distance and $L_1$ distance. The log-Sobolev inequalities and Poincar{\'e} inequalities for non-gradient SDEs are derived. All proofs and derivations are shown in section \ref{sec3} and \ref{sec31}. An analytical example is presented in section \ref{sec4}. 

\section{Main result}\label{sec2}
In this section, we present the main result of this paper. 

\subsection{Notations} 
Consider a SDE in a Euclidean connected compact domain $\Omega\subset\mathbb{R}^d$ by  \begin{equation}\label{a}
d X_t=b(X_t)dt+\sqrt{2}dB_t,
\end{equation}
where $X_t\in\Omega$ is a stochastic process, $B_t$ is a standard Brownian motion in $\Omega$, and $b\colon\mathbb{R}^d\rightarrow\mathbb{R}^d$ is a given smooth drift vector field. Suppose $\pi\in C^{\infty}(\Omega)$ is a given known invariant distribution for SDE \eqref{a}, then vector field $b$ satisfies  
\begin{equation*}
b(x)=\nabla\log\pi(x)-\gamma(x),
\end{equation*}
where $\gamma(x)\colon \Omega\rightarrow\mathbb{R}^d$ is a smooth vector field satisfying
\begin{equation*}
\nabla\cdot(\pi(x)\gamma(x))=0.     
\end{equation*}
The stochastic process \eqref{a} is reversible if and only if $\gamma(x)=0$; see \cite{Pav}. 

The Fokker-Planck equation of SDE \eqref{a} satisfies 
\begin{equation}\label{FPE}
\partial_tp(t,x)=-\nabla\cdot(p(t,x) b(x))+\Delta p(t,x),
\end{equation}
with Neumann boundary conditions and a given initial condition $p_0=p(0,x)$, such that 
\begin{equation*}
p_0\in L_1(\Omega),\quad p_0(x)\geq 0,\quad \int_\Omega p_0(x)dx=1.
\end{equation*}
Here $p_t=p(t,\cdot)$ is the probability density function of $X_t$. From the uniform
ellipticity condition, there exists a classical solution for \eqref{FPE}, where $p_t\in C^\infty(\mathbb{R}_+,\Omega)$ and $\lim_{t\rightarrow 0}p(t,x)=p_0(x)$. Here we can see that $\pi$ is the invariant distribution, since
\begin{equation*}
-\nabla\cdot(\pi(x)b(x))+\Delta \pi(x)=0.
\end{equation*}

The goal of this paper is to study the convergence behavior of probability density function $p(t,x)$ towards the invariant distribution $\pi(x)$. Here we are interested in the case that the invariant distribution $\pi$ has an analytical formula. 

To do so, we introduce the following Lyapunov functional in probability density space to study equation \eqref{FPE}. Given a probability density function $p$, denote a functional $\mathcal{I}$ to measure the difference between $p$ and $\pi$ by 
\begin{equation*}
\mathcal{I}(p\|\pi)=\int_\Omega \|\nabla \log\frac{p(x)}{\pi(x)}\|^2p(x)dx.
\end{equation*}
Here if $p=\pi$, then $\log \frac{p}{\pi}=0$ and $\mathcal{I}(p\|\pi)=0$. 
In literature, the functional $\mathcal{I}$ is known as the {\em relative Fisher information functional}. 
Shortly, we formulate the explicit condition to guarantee the exponential decay of relative Fisher information $\mathcal{I}$ along the Fokker-Planck equation \eqref{FPE}. And we establish a Poincar{\'e} inequality for non-reversible SDEs. 
\subsection{Main results}
We are now ready to present the condition to characterize the exponential convergence result of relative Fisher information. 

We first present a matrix function, whose lower bound of smallest eigenvalue will guarantee the convergence rate. Denote a vector function $\gamma=(\gamma_i)_{i=1}^d\colon \Omega\rightarrow\mathbb{R}^d$ by
\begin{equation}\label{G}
\gamma_i(x)=\frac{\partial}{\partial{x_i}}\log \pi(x)-b_i(x). 
\end{equation}
\begin{definition}
Denote a matrix function $\mathfrak{R}=(\mathfrak{R}_{ij})_{1\leq i,j\leq d}\colon \Omega\rightarrow \mathbb{R}^{d\times d}$, such that
\begin{equation}\label{R}
\mathfrak{R}(x)=-\nabla^2\log \pi(x)-A,
\end{equation}
where $A=(A_{ij})_{1\leq i,j\leq d}\colon \Omega \rightarrow\mathbb{R}^{d\times d}$ is a symmetric matrix function satisfying 
\begin{equation*}
A_{ij}(x)=\begin{cases}
\frac{3-2d}{8}\gamma_i(x)^2-\frac{1}{8}\sum_{k=1}^d\gamma_k(x)^2+\gamma_i(x)\frac{\partial}{\partial x_i}\log \pi(x)&\textrm{if $i=j$;}\\
\frac{3-2d}{8}\gamma_i(x)\gamma_j(x)+\frac{1}{2}\big(\gamma_i(x)\frac{\partial}{\partial x_j}\log \pi(x)+\gamma_j(x)\frac{\partial}{\partial x_i}\log \pi(x)\big)&\textrm{if $i\neq j$.}
\end{cases}
\end{equation*}
In other words, \begin{equation*}
\mathfrak{R}_{ij}(x)=
-\frac{\partial^2}{\partial x_i\partial x_j}\log \pi(x)-A_{ij}(x), 
\end{equation*}
for any $1\leq i, j\leq d$.
\end{definition}

We next present the main result of this paper. In a word, if the smallest eigenvalue of $\mathfrak{R}(x)$ is bounded below by a positive constant, then the relative Fisher information converges to zero 
exponentially fast.   
\begin{theorem}[Fisher information dissipation]\label{thm1}
Suppose that there exists a constant $\lambda> 0$, such that 
\begin{equation}\label{C}
\mathfrak{R}(x) \succeq \lambda \mathbb{I}, \quad \textrm{for any $x\in\Omega$},
\end{equation}
where $\mathbb{I}\in\mathbb{R}^{d\times d}$ is an identity matrix, then 
\begin{equation*}
\mathcal{I}(p_t\|\pi)\leq  e^{-2\lambda t}\mathcal{I}(p_0\|\pi),
\end{equation*}
where $p_t$ is the solution of Fokker-Planck equation \eqref{FPE} and $p_0$ is the given initial distribution. 
\end{theorem}
We postpone the proof of Theorem \ref{thm1} in next section. Using Theorem \ref{thm1}, we prove the classical log-Sobolev inequality under the condition \eqref{C}. Define the 
Kullback--Leibler (KL) divergence, also named relative entropy, between $p$ and $\pi$ by 
\begin{equation*}
\mathrm{D}_{\mathrm{KL}}(p\|\pi)=\int_\Omega p(x)\log \frac{p(x)}{\pi(x)}dx.     
\end{equation*}
\begin{corollary}\label{col2}
Assume condition \eqref{C} holds. Then the log-Sobolev inequality with constant $\lambda$ (LSI($\lambda$)) holds
\begin{equation*}
\mathrm{D}_{\mathrm{KL}}(p\|\pi) \leq \frac{1}{2\lambda}\mathcal{I}(p\|\pi),
\end{equation*}
for any smooth probability density function $p$. 
\end{corollary}

Using the above convergence result and log-Sobolev inequalities, one can derive different decay results under other functionals.
\begin{corollary}\label{col3}
Assume condition \eqref{C} holds. Then there are following decay results. 
\begin{itemize}
\item[(i)] {KL divergence decay:}
\begin{equation*}
\mathrm{D}_{\mathrm{KL}}(p_t\|\pi)\leq  \mathrm{D}_{\mathrm{KL}}(p_0\|\pi)e^{-2\lambda t}.
\end{equation*}
\item[(ii)]{Wasserstein-2 distance decay:}
\begin{equation*}
W(p_t,\pi)\leq \sqrt{\frac{2}{\lambda}\mathrm{D}_{\mathrm{KL}}(p_0\|\pi)}e^{-\lambda t}.
\end{equation*}
Here $W(p_t, \pi)$ is the Wasserstein-2 distance between $p_t$ and $\pi$, defined by 
\begin{equation*}
W(p_t, \pi)^2=\inf_{\pi}\int_{\Omega}\int_{\Omega}\|x-y\|^2\Pi(x,y)dxdy,
\end{equation*}
where the infimum is taken among all joint measures $\Pi$ with marginals $p_t$, $\pi$. 
\item[(iii)]{$L^1$ distance decay:}
\begin{equation*}
\int_\Omega \|p(t,x)-\pi(x)\| dx\leq \sqrt{2  \mathrm{D}_{\mathrm{KL}}(p_0\|\pi) }e^{-\lambda t}.
\end{equation*}
\end{itemize}
\end{corollary}

\begin{theorem}[Non-reversible Poincar{\'e} inequality]\label{thm2}
Assume condition \eqref{C} holds. Then the following inequality holds. 
\begin{equation*}
\int_\Omega \|h(x)-\int_\Omega h(y)\pi(y)dy\|^2 \pi(x)dx\leq \frac{1}{\lambda}\int_\Omega\Big(\|\nabla h(x)\|^2-h(x)(\nabla h(x), \gamma(x))\Big)\pi(x)dx, 
\end{equation*}
for any $h\in C^\infty(\Omega)$.
\end{theorem}
\begin{remark}[Comparison with entropy dissipations for reversible SDEs]
If $b(x)=-\nabla U(x)$, where $U\in C^\infty(\Omega)$ is a given function, then our dissipation rate recovers the classical entropy dissipation results. In other words, 
\begin{equation*}
\pi(x)=\frac{1}{Z}e^{-U(x)},\quad \textrm{where $Z$ is a normalization constant,}
\end{equation*}
and $\mathfrak{R}(x)=-\nabla^2\log \pi(x)=\nabla^2U(x)$. In this case, our condition \eqref{C} forms $\nabla^2U(x)\succeq \lambda \mathbb{I}$. The dissipation result in Theorem \ref{thm1} forms the classical entropy dissipation result \cite{Villani2009_optimal}. We emphasize that the entropy dissipation result also holds for non-gradient stochastic dynamics. And we provide a condition by $\mathfrak{R}(x)=-\nabla^2\log \pi(x)-A$, where the formulation of $A$ depends on the non-gradient term explicitly. 
\end{remark}

 \begin{remark}[Comparison with classical Gamma calculus]
We remark that our result is different from the classical Gamma calculus studied in \cite{BE, BE2008, BGH2019,CG2017, K} for general drift vector fields. In other words,
\begin{equation*}
\mathfrak{R}_{ij}(x)\neq -\frac{1}{2}(\frac{\partial}{\partial x_i}b_j(x)+\frac{\partial}{\partial x_j}b_i(x)).    
\end{equation*}
In classical Gamma calculus \cite{BE}, they consider the dissipation in the sense of semi-groups. 
It evolves the expectation of other testing function in the transition kernel. Our results of Gamma calculus characterize the dissipation in the sense of Fisher information. The resulting tensors are different in term of formulations for non-reversible SDEs. These two methods have different purposes. The Bakry-{\'E}mery's calculus works for general $b$ without any assumptions on $\pi$, while ours do require to know the analytical formulation of $\pi$. 
 \end{remark}
 
 \begin{remark}[Comparison with Arnold--Carlen tensor]
 We remark that Arnold--Carlen derive the other generalized Bakry--{\'E}mery tensor \cite{AC}; see also related studies in \cite{ACJ, FJ}. In details, 
 \begin{equation*}
 \mathfrak{R}_{\mathrm{AC}}=-\nabla^2\log \pi-\frac{1}{2}(\nabla \gamma+\nabla\gamma^{\ts}). 
 \end{equation*} 
 Our formulation in the weak formulation coincides with the Arnold--Carlen's tensor. We remark that the lower bound of smallest eigenvalue of $\mathfrak{R}_{\mathrm{AC}}$ also implies all dissipation results and inequalities in this paper. We provide geometric connections in subsection \ref{ACT}. Later on from a numerical example, we demonstrate that our tensor focuses on the local spatial domain convergence behavior while the Arnold-Carlen tensor works on the convergence property in global spatial domain. The combination usage of all these Gamma calculuses for studying general stochastic dynamics are left for future works. 
  \end{remark}
 
 \begin{remark}
We note that our results provide the other way to derive explicit bounds for classical log-Sobolev inequalities of a given invariant distribution \cite{AC, Gross}. A fact is that each gradient or non-gradient stochastic system is associated with a tensor $\mathfrak{R}$ and its positive lower bound $\lambda$ (if it exists) for the associated log-Sobolev inequality. Here the bound $\lambda$ can be used to prove the exponential decay result as follows:  
\begin{equation*}
\mathfrak{R}\succeq \lambda \mathbb{I} \Rightarrow \textrm{$\mathcal{I}(p_t\|\pi)$ decay}\Rightarrow \textrm{LSI($\lambda$)}\Rightarrow \textrm{$\mathrm{D}_{\mathrm{KL}}(p_t\|\pi)$ decay}\Rightarrow \textrm{$L_1(p_t\|\pi)$ decay.}      
\end{equation*}
\end{remark}
\begin{remark} It is worth mentioning that our Poincar{\'e} inequality recovers the classical Poincar{\'e} inequality when $\gamma=0$. 
Similarly, our proof for non-reversibility Poincar{\'e} inequality can also be generalized to the ones using the dimension of domain. We leave details of this direction in future works. 
\end{remark}
 
\section{Proof of Theorem \ref{thm1}}\label{sec3}
In this section, we present the proof of this paper. 
\subsection{Motivation}
We provide the motivation for our main result, especially Theorem \ref{thm1}. It is a Lyapunov method on probability density space. Our derivation is given by the following three steps. 
\begin{itemize}
\item We first compute the derivative along with the dynamics \eqref{FPE}. 
\begin{equation*}
\frac{d}{dt}\mathcal{I}(p_t\|\pi)=-2\cdot \textrm{Bilinear form}.
\end{equation*}
Here the bilinear form is the quadratic term functional involving with function $\log\frac{p_t}{\pi}$. 
\item We second complete the square in the first step, derive a new bilinear form and obtain the smallest eigenvalue from it, i.e. 
\begin{equation*}
\textrm{Bilinear form}\geq \lambda \mathcal{I}(p_t\|\pi). 
\end{equation*}
\item We last finish the proof by using the Grownwall's inequality for  
\begin{equation*}
\frac{d}{dt}\mathcal{I}(p_t\|\pi)\leq -2\lambda\mathcal{I}(p_t\|\pi).
\end{equation*}
\end{itemize}

We next present the above proofs in details. To do so, we apply the following notations.

\noindent{\textbf{Notations}:} Denote $\tilde L\colon C^{\infty}(\Omega)\rightarrow C^\infty(\Omega)$ as follows: for any $h\in C^\infty(\Omega)$, define
\begin{equation*}
\tilde L h= (\nabla\log \pi, \nabla h)+\Delta h. 
\end{equation*}
Denote $\tilde L^*\colon C^{\infty}(\Omega)\rightarrow C^{\infty}(\Omega)$ by 
\begin{equation*}
\tilde L^* p=-\nabla\cdot(p\nabla \log \pi)+\Delta p.
\end{equation*}
Clearly, we have the following dual relation between $\tilde L$ and $\tilde L^*$, i.e.
\begin{equation}\label{dual}
\int_\Omega \tilde Lh(x) p(x)dx= \int_\Omega h(x)L^*p(x)dx,
\end{equation}
for any $h\in C^\infty(\Omega)$.

For any functions $f\in C^{\infty}(\Omega)$, denote the following three operators. 
\begin{itemize}
\item[(i)] Gamma one operator: Denote $\Gamma_1\colon C^\infty(\Omega)\times C^\infty(\Omega)\rightarrow\mathbb{R}$ by   
\begin{equation*}
\Gamma_1(f,f)=(\nabla f, \nabla f).
\end{equation*}
\item[(ii)] Gamma two operator: Denote $\tilde \Gamma_2\colon C^\infty(\Omega)\times C^\infty(\Omega)\rightarrow\mathbb{R}$ by
\begin{equation*}
\tilde \Gamma_2(f,f)=\frac{1}{2}\tilde L\Gamma_1(f,f)-\Gamma_1(\tilde L f, f).
\end{equation*}
\item[(iii)] Information Gamma operator: Denote $\Gamma_{\mathcal{I}}\colon C^\infty(\Omega)\times C^\infty(\Omega)\rightarrow\mathbb{R}$ by 
\begin{equation*}
\Gamma_{\mathcal{I}}(f,f)=-\frac{1}{2}\big(\gamma, \nabla\Gamma_1(f, f)\big)+\tilde Lf\cdot (\nabla f, \gamma), 
\end{equation*}
where vector field $\gamma\colon \Omega\rightarrow\mathbb{R}^d$ is defined by \eqref{G}, i.e. 
\begin{equation*}
\gamma(x)=\nabla\log \pi(x)-b(x). 
\end{equation*}
\end{itemize}
\subsection{Dissipation of relative Fisher information}
We first compute the dissipation of relative Fisher information along with Fokker-Planck equation \eqref{FPE}. 
\begin{lemma}\label{lem1}
The following equality holds. Denote $p_t=p(t,x)$ as the solution of Fokker-Planck equation \eqref{FPE}, then
\begin{equation*}
\frac{d}{dt}\mathcal{I}(p_t\|\pi)=-2\int_\Omega\Big[ \tilde \Gamma_2(\log\frac{p_t}{\pi}, \log\frac{p_t}{\pi})+\Gamma_{\mathcal{I}}(\log\frac{p_t}{\pi}, \log\frac{p_t}{\pi})\Big]p_tdx.
\end{equation*}
\end{lemma}
\begin{proof}
Before the computation, we notice that there exists a global solution for \eqref{FPE} from the uniform elliptical condition. And by perturbation argument in the appendix of \cite{OV}, we can carry out the proof as follows. 

We compute the derivative of $t$ for relative Fisher information along the Fokker-Planck equation \eqref{FPE}. 
\begin{equation}\label{mainc}
\begin{split}
\frac{d}{dt}\mathcal{I}(p_t\|\pi)=&\frac{d}{dt}\int_\Omega \|\nabla\log\frac{p_t}{\pi}\|^2p_t dx\\
=&\int_\Omega \frac{\partial}{\partial t}\Big(\|\nabla\log\frac{p_t}{\pi}\|^2\Big)p_t dx+\int_\Omega \|\nabla\log\frac{p_t}{\pi}\|^2 \frac{\partial}{\partial t}p_t dx\\
=&2\int_\Omega\Big(\nabla\log\frac{p_t}{\pi}, \nabla\partial_t\log\frac{p_t}{\pi}\Big)p_t dx+\int_\Omega \|\nabla\log\frac{p_t}{\pi}\|^2 \partial_tp_t dx\\
=&2\int_\Omega\Big(\nabla\log\frac{p_t}{\pi}, \nabla\frac{\partial_t p_t}{p_t}\Big)p_t dx+\int_\Omega \|\nabla\log\frac{p_t}{\pi}\|^2 \partial_tp_t dx\\
=&-2\int_\Omega\frac{1}{p_t} \nabla\cdot (p_t\nabla\log\frac{p_t}{\pi})\partial_t p_tdx+\int_\Omega \|\nabla\log\frac{p_t}{\pi}\|^2 \partial_tp_t dx\\
=&-2\int_\Omega \Big((\frac{\nabla p_t}{p_t}, \nabla\log\frac{p_t}{\pi})+\Delta\log\frac{p_t}{\pi}\Big)\partial_t p_tdx+\int_\Omega \|\nabla\log\frac{p_t}{\pi}\|^2 \partial_tp_t dx\\
=&-2\int_\Omega \Big((\nabla\log p_t, \nabla\log\frac{p_t}{\pi})+\Delta\log\frac{p_t}{\pi}\Big)\partial_t p_tdx+\int_\Omega \|\nabla\log\frac{p_t}{\pi}\|^2 \partial_tp_t dx
\end{split}
\end{equation}
\begin{equation*}
    \begin{split}
=&-2\int_\Omega \Big(\|\nabla\log \frac{p_t}{\pi}\|^2+(\nabla \log \pi, \nabla \log\frac{p_t}{\pi})+\Delta\log\frac{p_t}{\pi}\Big)\partial_t p_tdx\\
&\quad +\int_\Omega \|\nabla\log\frac{p_t}{\pi}\|^2 \partial_tp_t dx\\
=&-2\int_\Omega \Big((\nabla \log \pi, \nabla \log\frac{p_t}{\pi})+\Delta\log\frac{p_t}{\pi}\Big)\partial_t p_tdx\\
&-\int_\Omega \|\nabla\log\frac{p_t}{\pi}\|^2 \partial_tp_t dx,
\end{split}
\end{equation*}
where we apply $\nabla\cdot(p_t\nabla\log\frac{p_t}{\pi})=(\nabla p_t, \nabla\log\frac{p_t}{\pi})+p_t\Delta\log\frac{p_t}{\pi}$ in the sixth equality and $\frac{\nabla p_t}{p_t}=\nabla \log p_t$ in the seventh equality. 

We next apply some notations to simply the formulations in \eqref{mainc}. Notice 
\begin{equation*}
\partial_t p_t=-\nabla\cdot(p_t b)+\Delta p_t=\tilde L^*p_t+\nabla\cdot(p_t \gamma).
\end{equation*}
Hence
\begin{equation*}
\begin{split}
-\frac{1}{2}\frac{d}{dt}\mathcal{I}(p_t\|\pi)=&\int_\Omega \tilde L\log\frac{p_t}{\pi}\partial_t p_tdx+\frac{1}{2}\int_\Omega \Gamma_1(\log\frac{p_t}{\pi}, \log\frac{p_t}{\pi}) \partial_tp_t dx\\
=&\quad\int_\Omega \tilde L\log\frac{p_t}{\pi} \Big(\tilde L^*p_t+\nabla\cdot(p_t\gamma)\Big)dx\\
&+\frac{1}{2}\int_\Omega \Gamma_1(\log\frac{p_t}{\pi}, \log\frac{p_t}{\pi})\Big(\tilde L^*p_t+\nabla\cdot(p_t\gamma)\Big)dx\\
=&\quad\int_\Omega \tilde L\log\frac{p_t}{\pi}\tilde L^*p_t+\frac{1}{2}\Gamma_1(\log\frac{p_t}{\pi}, \log\frac{p_t}{\pi})\tilde L^*p_t dx\\
&+\int_\Omega \tilde L\log\frac{p_t}{\pi} \nabla\cdot(p_t\gamma) dx+\frac{1}{2}\int_\Omega \Gamma_1(\log\frac{p_t}{\pi}, \log\frac{p_t}{\pi})\nabla\cdot(p_t\gamma) dx.
\end{split}
\end{equation*}
In the last of derivation, we show the following claim. 

\noindent{\textbf{Claim}:}
The following equalities hold: 
\begin{itemize}
\item[(i)]
\begin{equation*}
\int_\Omega \Gamma_2(\log\frac{p_t}{\pi}, \log\frac{p_t}{\pi})p_t dx=\int_\Omega \tilde L\log\frac{p_t}{\pi}\tilde L^*p_t+\frac{1}{2}\Gamma_1(\log\frac{p_t}{\pi}, \log\frac{p_t}{\pi})\tilde L^*p_t dx.
\end{equation*}
\item[(ii)]
\begin{equation*}
\int_\Omega \Gamma_{\mathcal{I}}(\log\frac{p_t}{\pi}, \log\frac{p_t}{\pi})p_t dx=\int_\Omega \tilde L\log\frac{p_t}{\pi} \nabla\cdot(p_t\gamma) dx+\frac{1}{2}\int_\Omega \Gamma_1(\log\frac{p_t}{\pi}, \log\frac{p_t}{\pi})\nabla\cdot(p_t\gamma) dx.
\end{equation*}
\end{itemize}
If the claim is true, our result is shown. Our goal is to present the proof. 

\noindent \textbf{\em Proof of Claim}
(i) The first part of proof follows the classical dissipation result \cite{OV, Villani2009_optimal} etc, widely used in $L^2$--Wasserstein gradient flows and their generalizations \cite{LiG2}. For the completeness of this paper, we still present it here. On the one hand, from the dual relation \eqref{dual} between $\tilde L$ and $\tilde L^*$, we have  
\begin{equation*}
\int_\Omega \frac{1}{2}\Gamma_1(\log\frac{p_t}{\pi}, \log\frac{p_t}{\pi})\tilde L^*p_t dx=\frac{1}{2}\int_\Omega \tilde L\Gamma_1(\log\frac{p_t}{\pi}, \log\frac{p_t}{\pi}) p_t dx.
\end{equation*}
On the other hand, 
\begin{equation*}
\begin{split}
\int_\Omega \tilde L\log\frac{p_t}{\pi}\tilde L^*p_t=&\int_\Omega \tilde L\log\frac{p_t}{\pi}\nabla\cdot(p_t\nabla\log\frac{p_t}{\pi})dx\\
=&-\int_\Omega (\nabla \tilde L\log\frac{p_t}{\pi}, \nabla \log\frac{p_t}{\pi})p_t dx\\
=&-\int_\Omega\Gamma_1(\tilde L\log\frac{p_t}{\pi}, \log\frac{p_t}{\pi})p_tdx,
\end{split}
\end{equation*}
where the second equality holds by the integration by part formula. Combining above derivations, we obtain the result. 

(ii) The second part of derivation majorly depends on the non-gradient vector field $\gamma$. On the one hand, 
\begin{equation*}
\begin{split}
\int_\Omega \tilde L\log\frac{p_t}{\pi} \nabla\cdot(p_t\gamma) dx=&\int_\Omega \tilde L\log\frac{p_t}{\pi}\Big((\nabla p_t, \gamma)+p_t\nabla\cdot \gamma\Big)dx\\
=&\int_\Omega \tilde L\log\frac{p_t}{\pi}\Big((\nabla \log p_t, \gamma)+\nabla\cdot \gamma\Big)p_tdx\\
=&\int_\Omega \tilde L\log\frac{p_t}{\pi}\Big((\nabla \log \frac{p_t}{\pi}, \gamma)+(\nabla \log \pi, \gamma)+\nabla\cdot \gamma\Big)p_tdx\\
=&\int_\Omega \tilde L\log\frac{p_t}{\pi}\Big((\nabla \log \frac{p_t}{\pi}, \gamma)+\pi\nabla\cdot(\pi \gamma)\Big)p_tdx\\
=&\int_\Omega \tilde L\log\frac{p_t}{\pi}(\nabla \log \frac{p_t}{\pi}, \gamma) p_tdx,
\end{split}
\end{equation*}
where the last equality uses  the fact that $\pi$ is an invariant distribution, i.e. 
\begin{equation*}
\nabla\cdot( \pi \gamma)=0. 
\end{equation*}
On the other hand, we have 
\begin{equation*}
\begin{split}
\frac{1}{2}\int_\Omega \Gamma_1(\log\frac{p_t}{\pi}, \log\frac{p_t}{\pi})\nabla\cdot(p_t\gamma) dx=-\frac{1}{2}\int_\Omega \Big(\nabla \Gamma_1(\log\frac{p_t}{\pi}, \log\frac{p_t}{\pi}), \gamma\Big)p_t dx,
\end{split}
\end{equation*}
where the equality holds by the integration by parts formula. Combining the above two steps, we finish the proof. 
\end{proof}

\subsection{Fisher information induced Gamma calculus}
We next derive the dissipation rate of relative Fisher information functional. 
\begin{lemma}\label{lem2}
Given any test function $f\in C^{\infty}(\Omega)$, the following equalities hold.  
\begin{equation*}
\begin{split}
&\Gamma_2(f,f)+\Gamma_{\mathcal{I}}(f,f)\\
=&\|\nabla^2 f\|^2_{\mathrm{F}}-\nabla^2\log \pi(\nabla f, \nabla f)-(\gamma, \nabla^2f\nabla f)+\Delta f(\nabla f, \gamma)+(\nabla f, \nabla \log\pi)(\nabla f, \gamma)\\
=&\|\mathfrak{Hess}f\|_{\mathrm{F}}^2+\mathfrak{R}(\nabla f,\nabla f).
\end{split}
\end{equation*}
Here $\|\cdot\|_{\mathrm{F}}$ is the Frobenius norm, $\mathfrak{R}$ is the matrix function defined in \eqref{R}, and $\mathfrak{Hess}f=(\mathfrak{Hess}f_{ij})_{1\leq i,j\leq d}\colon \Omega\rightarrow\mathbb{R}^{d\times d}$ is the matrix function defined by
\begin{equation*}
\mathfrak{Hess}f(x)_{ij}=
\begin{cases}
\frac{\partial^2}{\partial x_i\partial x_i}f(x)+\frac{1}{2}\sum_{k=1}^d\frac{\partial}{\partial x_k}f(x)\gamma_k(x)-\frac{1}{2}\frac{\partial}{\partial x_i}f(x)\gamma_i(x) &\textrm{if $i=j$;}\\
\frac{\partial^2}{\partial x_i\partial x_j} f(x)-\frac{1}{4}\Big(\gamma_i(x)\frac{\partial}{\partial x_j}f(x)+\gamma_j(x)\frac{\partial}{\partial x_i} f(x)\Big) &\textrm{if $j\neq i$.}
\end{cases}
\end{equation*}
 \end{lemma}
\begin{remark}
Lemma \ref{lem2} is a generalization of Gamma calculus. If $\gamma=0$, it recovers the standard Bakry-{\'E}mery Gamma calculus
 \begin{equation*}
\Gamma_2(f,f)= \|\nabla^2 f\|^2_{\mathrm{F}}-\nabla^2\log \pi(\nabla f, \nabla f). 
 \end{equation*} 
 We notice that our dissipation rate is derived from the calculus based on Fisher information. For this reason, we call this derivation method the {\em information Gamma calculus}. 
\end{remark}
\begin{proof}
The proof is to complete a square for a quadratic form functional of $f$. On the one hand, 
\begin{equation}\label{step1}
\begin{split}
\Gamma_2(f,f)=&\frac{1}{2}\tilde L\Gamma_1(f,f)-\Gamma_1(\tilde L f, f)\\
=&\frac{1}{2}\Delta\|\nabla f\|^2-(\nabla \Delta f, \nabla f)+\frac{1}{2}(\nabla\log \pi, \nabla \|\nabla f\|^2)-(\nabla (\nabla\log \pi, \nabla f), \nabla f)\\
=&\|\nabla^2 f\|^2_{\mathrm{F}}+\nabla^2 f(\nabla^2\log \pi, \nabla f)-\nabla^2f(\nabla^2\log \pi, \nabla f)-\nabla^2\log \pi(\nabla f,\nabla f)\\
=&\|\nabla^2 f\|^2_{\mathrm{F}}-\nabla^2\log \pi(\nabla f, \nabla f),
\end{split}
\end{equation}
where the last equality applies the Bochner's formula, i.e. 
\begin{equation*}
\frac{1}{2}\Delta\|\nabla f\|^2-(\nabla \Delta f, \nabla f)=\|\nabla^2 f\|_{\mathrm{F}}^2. 
\end{equation*}
On the other hand, \begin{equation}\label{step2}
\begin{split}
\Gamma_{\mathcal{I}}(f,f)=&-\big(\gamma, \nabla\Gamma_1(f, f)\big)+\tilde Lf\cdot (\nabla f, \gamma)\\ 
=&-\big(\gamma, \nabla\|\nabla f\|^2\big)+\big(\Delta f+(\nabla f, \nabla \log \pi)\big) (\nabla f, \gamma)\\
=&-\nabla^2f(\gamma, \nabla f)+\Delta f(\nabla f, \gamma)+\big(\nabla f, \nabla \log \pi\big)(\nabla f, \gamma).
\end{split}
\end{equation}
Combining \eqref{step1} and \eqref{step2} together, we have
\begin{equation}\label{step3}
\begin{split}
\Gamma_2(f,f)+\Gamma_{\mathcal{I}}(f,f)=&\quad\|\nabla^2 f\|^2_{\mathrm{F}}-\nabla^2\log \pi(\nabla f, \nabla f)\\
&-\nabla^2f(\gamma, \nabla f)+\Delta f(\nabla f, \gamma)+(\nabla f, \nabla \log \pi)(\nabla f, \gamma).
\end{split}
\end{equation}
Finally, we derive the main result by completing the square in formulation \eqref{step3}. For simplicity of notations, denote 
\begin{equation*}
\partial_if=\frac{\partial}{\partial x_i}f,\quad \partial_{ij}f=\frac{\partial^2}{\partial x_i\partial x_j}f, \quad \partial_{ij}\log \pi=\frac{\partial^2}{\partial x_i\partial x_j}\log \pi.  
\end{equation*}
Under above notations, we reformulate \eqref{step3} explicitly as follows:
\begin{equation*}
\begin{split}
&\Gamma_2(f,f)+\Gamma_{\mathcal{I}}(f,f)\\=&\quad \sum_{i=1}^d|\partial_{ii}f|^2+2\sum_{1\leq i<j\leq d}|\partial_{ij}f|^2-\sum_{i=1}^d\sum_{j=1}^d\partial_{ij}\log \pi\partial_i f\partial_j f\\
&-\sum_{i=1}^d\sum_{j=1}^d \partial_{ij}f\gamma_i\partial_jf+(\sum_{i=1}^d\partial_{ii}f)(\sum_{j=1}^d\partial_jf\gamma_j) +(\sum_{i=1}^d\partial_i f\partial_i\log \pi)(\sum_{j=1}^d\partial_jf\gamma_j)\\
=&\quad\sum_{i=1}^d\Big(\partial_{ii}f-\frac{1}{2}\gamma_i\partial_if+\frac{1}{2}\sum_{j=1}^d\partial_jf\gamma_j\Big)^2+2\sum_{1\leq i<j\leq d}\Big(\partial_{ij}f-\frac{1}{4}(\gamma_i\partial_j f+\gamma_j\partial_i f)\Big)^2\\
&-\frac{1}{4}\sum_{i=1}^d\Big(\gamma_i\partial_if-\sum_{j=1}^d\partial_jf\gamma_j\Big)^2-\frac{1}{8}\sum_{1\leq i<j\leq d}\Big(\gamma_i\partial_jf+\gamma_j\partial_i f\Big)^2\\
&-\sum_{i=1}^d\sum_{j=1}^d\partial_{ij}\log \pi\partial_i f\partial_j f+(\sum_{i=1}^d\partial_i f\partial_i\log \pi)(\sum_{j=1}^d\partial_jf\gamma_j)\\
=&\quad\sum_{i=1}^d\Big(\partial_{ii}f+\frac{1}{2}\sum_{j\neq i}\partial_jf\gamma_j\Big)^2+2\sum_{1\leq i<j\leq d}\Big(\partial_{ij}f-\frac{1}{4}(\gamma_i\partial_j f+\gamma_j\partial_i f)\Big)^2\\
&-\frac{1}{4}\sum_{i=1}^d\Big(\sum_{j\neq i}\partial_jf\gamma_j\Big)^2-\frac{1}{8}\sum_{1\leq i<j\leq d}\Big(\gamma_i\partial_jf+\gamma_j\partial_i f\Big)^2\\
&-\sum_{i=1}^d\sum_{j=1}^d\Big(\partial_{ij}\log \pi-\partial_i\log \pi \gamma_j\Big)\partial_i f\partial_j f\\
=&\quad\sum_{i=1}^d\Big(\partial_{ii}f+\frac{1}{2}\sum_{j\neq i}\partial_jf\gamma_j\Big)^2+2\sum_{1\leq i<j\leq d}\Big(\partial_{ij}f-\frac{1}{4}(\gamma_i\partial_j f+\gamma_j\partial_i f)\Big)^2\\
&-\frac{1}{4}\sum_{i=1}^d\Big(\sum_{j\neq i}\partial_jf\gamma_j\Big)^2-\frac{1}{8}\sum_{1\leq i<j\leq d}\Big(\gamma_i^2(\partial_jf)^2+\gamma_j^2(\partial_i f)^2+2\gamma_i\gamma_j\partial_if\partial_jf\Big)\\
&-\sum_{i=1}^d\sum_{j=1}^d\Big(\partial_{ij}\log \pi-\frac{1}{2}(\partial_i\log \pi \gamma_j+\partial_j\log \pi \gamma_i)\Big)\partial_i f\partial_j f.
\end{split}
\end{equation*}
Hence
\begin{equation*}
    \begin{split}
\Gamma_2(f,f)+\Gamma_{\mathcal{I}}(f,f)=&\|\mathfrak{Hess}f\|_{\textrm{HS}}^2+\mathfrak{R}(\nabla f, \nabla f).
\end{split}
\end{equation*}

\end{proof}

\subsection{Main proof}
We are now ready to prove the main result. 
\begin{proof}[Proof of Theorem \ref{thm1}]
We apply the Lyapunov method in probability density space. From Lemma \ref{lem1}, we have 
\begin{equation*}
\begin{split}
\frac{d}{dt}\mathcal{I}(p_t)=&-2\int_\Omega \Big[\Gamma_2(\log\frac{p_t}{\pi},\log\frac{p_t}{\pi})+\Gamma_{\mathcal{I}}(\log\frac{p_t}{\pi},\log\frac{p_t}{\pi})\Big]p_t dx\\
=&-2\int_\Omega \Big[\|\mathfrak{Hess}\log\frac{p_t}{\pi}\|_{\mathrm{F}}^2+\mathfrak{R}(\nabla \log\frac{p_t}{\pi},\nabla\log\frac{p_t}{\pi})\Big] p_t dx\\
\leq &-2\lambda \int_\Omega\Gamma_1(\log\frac{p_t}{\pi}, \log\frac{p_t}{\pi})dx\\
=&-2\lambda \mathcal{I}(p_t). 
\end{split}
\end{equation*}
where the second inequality holds from Lemma \ref{lem2} with $f=\log\frac{p}{\pi}$, and the third inequality comes from condition \eqref{C}, which implies 
\begin{equation*}
\Gamma_2(f,f)+\Gamma_{\mathcal{I}}(f,f)\geq \lambda \Gamma_1(f,f). 
\end{equation*}
From the Gronwall inequality, we prove the result for the exponential convergence of $\mathcal{I}$. 
\end{proof}

\begin{proof}[Proof of Corollary \ref{col2}]
We first prove the following statement. 

\noindent{\textbf{Claim}:}
\begin{equation}\label{ed}
\frac{d}{dt}\mathrm{D}_{\mathrm{KL}}(p_t\|\pi)=-\mathcal{I}(p_t\|\pi),    
\end{equation}
where $p_t$ is a solution for Fokker-Planck equation \eqref{FPE}. In literature, we notice that the dissipation of KL divergence equals to the negative Fisher information. In other words, \eqref{ed} holds for the gradient system $b=-\nabla \log \pi$. Here we prove that \eqref{ed} is also true for non-gradient system. 
\begin{proof}[Proof of Claim]
Notice 
\begin{equation}\label{derive}
\begin{split}
 \frac{d}{dt}\mathrm{D}_{\mathrm{KL}}(p_t\|\pi)=&\frac{d}{dt}\int_\Omega p(t,x)\log\frac{p(t,x)}{\pi(x)}dx  \\
 =&\int_\Omega \partial_tp(t,x)\log\frac{p(t,x)}{\pi(x)}+p(t,x)\partial_t\log\frac{p(t,x)}{\pi(x)} dx\\
 =&\int_\Omega \partial_tp(t,x)\log\frac{p(t,x)}{\pi(x)}+p(t,x)\frac{\partial_tp(t,x)}{p(t,x)} dx\\
  =&\int_\Omega \partial_tp(t,x)\log\frac{p(t,x)}{\pi(x)}dx+\int_\Omega{\partial_tp(t,x)} dx.
\end{split}    
\end{equation}
Since $p_t$ satisfies the Fokker-Planck equation \eqref{FPE}, i.e. 
\begin{equation}\label{CC}
\begin{split}
\partial_tp(t,x)=&-\nabla\cdot(p(t,x) b(x))+\nabla\cdot(\nabla p(t,x))\\
=&-\nabla\cdot(p(t,x) b(x))+\nabla\cdot(p(t,x)\nabla \log p(t,x))\\
=&\nabla\cdot(p(t,x) (\nabla\log \pi(x)-b(x)))+\nabla\cdot(p(t,x)\nabla\log \frac{p(t,x)}{\pi(x)})\\
=&\nabla\cdot(p(t,x)\gamma(x))+\nabla\cdot(p(t,x)\nabla\log \frac{p(t,x)}{\pi(x)}),
\end{split}
\end{equation}
where we use the fact $p(t,x)\nabla \log p(t,x)=\nabla p(t,x)$ and denote $\gamma(x)=\nabla\log \pi(x)-b(x)$. Then 
\begin{equation*}
\int_\Omega \partial_t p(t,x)dx=0. 
\end{equation*}
Thus by substituting equation \eqref{CC} into \eqref{derive}, we have 
\begin{equation*}
\begin{split}
 \frac{d}{dt}\mathrm{D}_{\mathrm{KL}}(p_t\|\pi)=&\int_\Omega \partial_tp(t,x)\log\frac{p(t,x)}{\pi(x)}dx\\
 =&\quad \int_\Omega \nabla\cdot(p(t,x)\gamma(x)) \log\frac{p(t,x)}{\pi(x)}dx\\
 &+\int_\Omega \nabla\cdot(p(t,x)\nabla\log \frac{p(t,x)}{\pi(x)})\log\frac{p(t,x)}{\pi(x)}dx\\
 =&-\int_\Omega \Big(\nabla \log\frac{p(t,x)}{\pi(x)},  \gamma(x) \Big)p(x) dx\\
 &-\int_\Omega (\nabla\log \frac{p(t,x)}{\pi(x)}, \nabla \log\frac{p(t,x)}{\pi(x)}) p(t,x)dx,
\end{split}    
\end{equation*}
where the last equality holds by the integration by parts formulas. We also claim that 
\begin{equation*}
    \int_\Omega \Big(\nabla \log\frac{p(t,x)}{\pi(x)},  \gamma(x) \Big)p(t,x) dx=0. 
\end{equation*}
This is true since 
\begin{equation*}
\begin{split}
&\int_\Omega \Big(\nabla \log\frac{p(t,x)}{\pi(x)},  \gamma(x) \Big)p(x) dx\\
      =&\int_\Omega (\nabla \log p(t,x), \gamma(x))p(t,x)-(\nabla \log \pi(x), \gamma(x))p(t,x)dx\\
    =&\int_\Omega (\nabla  p(t,x), \gamma(x))-(\nabla \log \pi(x), \gamma(x))p(t,x)dx\\
    =&-\int_\Omega \Big(\nabla\cdot \gamma(x)-(\nabla \log \pi(x), \gamma(x))\Big)p(t,x)dx\\
    =& 0,
\end{split}
\end{equation*}
where we apply the integration by parts in the last equality and we use the fact that $\pi$ is the invariant measure, i.e. 
\begin{equation*}
0=\frac{1}{\pi(x)}\nabla\cdot(\pi(x)\gamma(x))=\nabla\cdot \gamma(x)-(\nabla \log \pi(x), \gamma(x)).   
\end{equation*}
This finishes the proof of the claim. 
\end{proof}
Based on Claim \eqref{ed} and the proof in Theorem \ref{thm1}, we have 
\begin{equation*}
\frac{d}{dt}\mathcal{I}(p_t\|\pi)\leq -2 \lambda \mathcal{I}(p_t\|\pi)=2\lambda \frac{d}{dt}\mathrm{D}_{\mathrm{KL}}(p_t\|\pi).   
\end{equation*}
By the integration w.r.t $[0,\infty)$ with $p(0,x)=p(x)$, we finish the proof of log-Sobolev inequality. 
\end{proof}

\begin{proof}[Proof of Corollary \ref{col3}]
The exponential decay of KL divergence follows the proof of Corollary \ref{col2}. Notice that 
\begin{equation*}
\frac{d}{dt}\mathrm{D}_{\textrm{KL}}(p_t\|\pi)=-\mathcal{I}(p_t\|\pi)\leq -2\lambda \mathrm{D}_{\mathrm{KL}}(p_t\|\pi).     
\end{equation*}
From the Gronwall inequality, we prove (i). We next apply the inequalities proved in Theorem 1 of \cite{OV} under our Corollary \ref{col2} and condition \eqref{C}. Notice that 
\begin{equation*}
W(p_t, \pi)\leq \sqrt{\frac{2}{\lambda}\mathrm{D}_{\mathrm{KL}}(p_t\|\pi)},
\end{equation*}
and 
\begin{equation*}
\int_\Omega \|p_t-\pi\|dx\leq \sqrt{2\mathrm{D}_{\mathrm{KL}}(p_t\|\pi)}.
\end{equation*}
Combining them with the result in Theorem \ref{thm1}, we finish the proof.
\end{proof}
\subsection{Connections with Arnold--Carlen tensor}\label{ACT}
In this subsection, we notice that there are reformulations of information Gamma calculus in the weak sense. In particular, it is also equivalent to the Arnold--Carlen generalized Bakry-{\'E}mery condition in the weak sense. 
\begin{proposition}[Weak formulation of information Gamma calculus]
The following equality holds. Denote $f(x)=\log\frac{p(x)}{\pi(x)}$, then 
\begin{equation*}
\begin{aligned}
\int_\Omega \Big(\Gamma_{2}(f, f)+\Gamma_{\mathcal{I}}(f,f)\Big)p(x)dx=\int_\Omega \Big(\|\nabla^2 f\|_{\mathrm{F}}^2+\mathfrak{R}_{\mathrm{AC}}(\nabla f,\nabla f)\Big)p(x)dx, 
\end{aligned}\
\end{equation*}
where $\mathfrak{R}_{\mathrm{AC}}(x)=(\mathfrak{R}_{AC}(x)_{ij})_{1\leq i,j\leq d}\in \mathbb{R}^{d\times d}$ is the Arnold--Carlen tensor \cite{AC} defined by 
\begin{equation*}
\mathfrak{R}_{AC}(x)_{ij}=-\frac{\partial^2}{\partial x_i\partial x_j}\log\pi(x)-\frac{1}{2}\Big(\frac{\partial}{\partial x_i}\gamma_j(x)+\frac{\partial}{\partial x_j}\gamma_i(x)\Big).
\end{equation*}
\end{proposition}
\begin{proof}
Here we only need to prove 
\begin{equation*}
\int_\Omega \Gamma_{\mathcal{I}}(f,f) p(x)dx=\int_\Omega \mathfrak{R}_{\mathrm{AC}}(\nabla f,\nabla f) p(x)dx.
\end{equation*}
Notice that $\nabla\cdot(\pi \gamma)=0$, hence we have
\begin{equation*}
\nabla\cdot(p\gamma)=p(\nabla\log\frac{p}{\pi}, \gamma)=p(\nabla f, \gamma).
\end{equation*}
Then
\begin{equation*}
\begin{split}
\int_\Omega \Gamma_{\mathcal{I}}(f,f) p dx=&\int_\Omega \Big\{\Delta f(\nabla f, \gamma)+(\nabla\log \pi, \nabla f)(\nabla f, \gamma)+\frac{1}{2}\Gamma_1(f,f)(\nabla f, \gamma)\Big\}p dx. 
\end{split}
\end{equation*}
We observe that 
\begin{equation*}
\begin{split}
\int_\Omega \Delta f (\nabla f, \gamma)pdx=&\int_\Omega -(\nabla f, \nabla(p(\nabla f, \gamma))) dx\\
=& -\int_\Omega \Big\{(\nabla f,\nabla \log p)(\nabla f, \gamma)+\nabla^2f(\nabla f, \gamma)+\nabla\gamma(\nabla f, \nabla f)\Big\}p dx\\
\end{split}
\end{equation*}
Hence 
\begin{equation*}
\begin{split}
&\int_\Omega \Gamma_{\mathcal{I}}(f,f) p dx\\=&\int_\Omega -\frac{1}{2}\Gamma_1(f,f)(\nabla f, \gamma)p-\Big\{\nabla^2f(\nabla f, \gamma)+\nabla\gamma(\nabla f, \nabla f)\Big\} pdx\\
=&\int_\Omega -\frac{1}{2}\Gamma_1(f,f)(\nabla p, \gamma)+\frac{1}{2}\Gamma_1(f,f)(\nabla \log\pi, \gamma)p-\Big\{\nabla^2f(\nabla f, \gamma)+\nabla\gamma(\nabla f, \nabla f)\Big\} pdx\\
=&\int_\Omega \frac{1}{2}\nabla\cdot(\Gamma_1(f,f)\gamma)p+\frac{1}{2}\Gamma_1(f,f)(\nabla\log \pi,\gamma)p-\Big\{\nabla^2f(\nabla f, \gamma)+\nabla\gamma(\nabla f, \nabla f)\Big\} pdx\\
=&\int_\Omega \nabla^2f(\nabla f, \gamma)p+\frac{1}{2}\Gamma_1(f,f)(\nabla\cdot\gamma+(\nabla\log \pi, \gamma))p-\Big\{\nabla^2f(\nabla f, \gamma)+\nabla\gamma(\nabla f, \nabla f)\Big\}p dx\\
=&-\int_\Omega \nabla\gamma(\nabla f, \nabla f)p dx. 
\end{split}
\end{equation*}
\end{proof}
\section{Non-reversible Poincar{\'e} inequality}\label{sec31}
In this section, we formulate the proof of non-reversible Poincar{\'e} inequality. This corresponds to the asymptotic formulation of information Gamma calculus. 

To do so, we first derive an equality, using the weak form of information Gamma calculus.
\begin{lemma}\label{lem2}
For any $\Phi\in C^\infty(\Omega)$, the following equality holds.
\begin{equation}\label{Yano}
\int_\Omega \Big(\Gamma_2(\Phi, \Phi)+\Gamma_{\mathcal{I}}(\Phi,\Phi)\Big) \pi(x)dx=\int_\Omega \tilde L \Phi(x) \Big(\tilde L\Phi(x)+(\nabla \Phi(x), \gamma(x))\Big)\pi(x) dx.
\end{equation}
\end{lemma}
\begin{proof}
Firstly, notice that $\gamma$ satisfies $\nabla\cdot\gamma=-(\nabla\log \pi, \gamma)$. Hence 
\begin{equation*}
\nabla\cdot(p\gamma)=p\Big((\nabla\log p, \gamma)+\nabla\cdot \gamma)=p(\nabla\log\frac{p}{\pi}, \gamma). 
\end{equation*}
From the derivation of Lemma \ref{lem1}, we have
\begin{equation}\label{h1}
\begin{split}
\int_\Omega \Gamma_{\mathcal{I}}(\log\frac{p}{\pi}, \log\frac{p}{\pi})p dx=&\int_\Omega \Big(\tilde L\log\frac{p}{\pi}+\frac{1}{2}\Gamma_1(\log\frac{p}{\pi}, \log\frac{p}{\pi})\Big)\nabla\cdot(p\gamma) dx\\
=&\int_\Omega \Big(\tilde L\log\frac{p}{\pi}+\frac{1}{2}\Gamma_1(\log\frac{p}{\pi}, \log\frac{p}{\pi})\Big)(\nabla\log\frac{p}{\pi}, \gamma)p dx.
\end{split}
\end{equation}
Given a constant $\epsilon\in \mathbb{R}$, denote 
\begin{equation}\label{h}
\log\frac{p(x)}{\pi(x)}=\epsilon \Phi(x).
\end{equation}
Clearly, $p(x)=\pi(x)e^{\epsilon \Phi(x)}$ and $p(x)=\pi(x)$ if $\epsilon=0$. Substituting \eqref{h} into \eqref{h1}, we have
\begin{equation*}
\epsilon^2\int_\Omega \Gamma_{\mathcal{I}}( \Phi, \Phi)p dx=\epsilon^2\int_\Omega \tilde L\Phi (\nabla \Phi, \gamma) p dx+\epsilon^3\int_\Omega \Gamma_1(\Phi, \Phi)(\nabla \Phi, \gamma) p dx. 
\end{equation*}
Let $\epsilon$ shrink to zero, then 
\begin{equation*}
\int_\Omega \Gamma_{\mathcal{I}}( \Phi, \Phi)\pi dx=\int_\Omega \tilde L \Phi (\nabla \Phi, \gamma) \pi dx.
\end{equation*}
Secondly, following the fact that  
\begin{equation*}
\begin{split}
\int_\Omega \Gamma_{2}(\Phi, \Phi)\pi dx=&\int_\Omega \frac{\Big(\nabla\cdot(\pi\nabla \Phi)\Big)^2}{\pi}dx \\
=&\int_\Omega \Big(\frac{\nabla\cdot(\pi\nabla\Phi)}{\pi}\Big)^2\pi dx\\
=&\int_\Omega (\tilde L\Phi)^2\pi dx. 
\end{split}
\end{equation*}
From the above two facts, we have
\begin{equation*}
\begin{split}
\int_\Omega\Big(\Gamma_{2}(\Phi, \Phi)+\Gamma_{\mathcal{I}}(\Phi, \Phi)\Big)\pi dx=&\int_\Omega \tilde L \Phi \Big(\tilde L\Phi+(\nabla \Phi, \gamma)\Big)\pi dx.
\end{split}
\end{equation*}
This finishes the proof.
\end{proof}
\begin{remark}
For reversible SDEs, Lemma \ref{lem2} shows the Hessian operator of relative entropy in $L^2$--Wasserstein space at the invariant distribution \cite{LiG2}. Here we extend this formula in general non-gradient flows. See related geometric reasons in appendix. 
\end{remark}
From Lemma \ref{lem2}, we are now ready to prove the non-reversible Poincar{\'e} inequality. 
\begin{proof}[Proof of Theorem \ref{thm2}]
From the condition \eqref{C}, we have 
\begin{equation}\label{C1}
\int_\Omega \Big(\Gamma_2(\Phi, \Phi)+\Gamma_{\mathcal{I}}(\Phi, \Phi)\Big) \pi dx \geq \lambda \int_\Omega \Gamma_1(\Phi, \Phi)\pi dx. 
\end{equation}
From Lemma \ref{lem2}, we have 
\begin{equation*}
\begin{split}
\int_\Omega\Big(\Gamma_{2}(\Phi, \Phi)+\Gamma_{\mathcal{I}}(\Phi, \Phi)\Big)\pi dx=&\int_\Omega \tilde L \Phi \Big(\tilde L\Phi+(\nabla \Phi,\gamma)\Big)\pi dx\\
=&\int_\Omega \frac{\nabla\cdot(\pi\nabla \Phi)}{\pi} \frac{\nabla\cdot(\pi\nabla\Phi)+(\nabla \Phi,\gamma)\pi}{\pi}\pi dx\\
=&\int_\Omega \nabla\cdot(\pi\nabla \Phi)\cdot \frac{1}{\pi}\cdot \Big(\nabla\cdot(\pi\nabla\Phi)+(\nabla \Phi,\gamma)\pi\Big) dx\\
=&\int_\Omega \Big(\Phi, g_\mathrm{W}^{-1}(\pi)\cdot g_F(\pi)\cdot (-L^*_{\pi,\gamma}) \Phi\Big) dx.
\end{split}
\end{equation*}
In above, we denote the following operators. Denote the $L^2$--Wasserstein metric tensor by $g_\mathrm{W}(\pi)=-\nabla\cdot(\pi\nabla)$, i.e. 
\begin{equation*}
    g_\mathrm{W}(\pi)^{-1}\Phi=-\nabla\cdot(\pi\nabla\Phi),
\end{equation*}
and $g_F(\pi)$ is the Fisher-Rao metric tensor defined by 
\begin{equation*}
g_{\mathrm{F}}(\pi)\sigma=\frac{\sigma(x)}{\pi(x)}-\int_\Omega \sigma(x)dx,
\end{equation*}
for any $\sigma\in C^{\infty}(\Omega)$. In addition, define 
\begin{equation*}
L^*_{\pi,\gamma}\Phi=\nabla\cdot(\pi\nabla \Phi)+(\nabla \Phi, \gamma)\pi,
\end{equation*}
Notice 
\begin{equation*}
\int_\Omega\Gamma_1(\Phi, \Phi)\pi dx= \int_\Omega \Phi(-\Delta_\pi \Phi) dx. 
\end{equation*}
We remark that the above operator operation is defined on the tangent space in probability density space. See related discussions in appendix.

Thus the condition \eqref{C1} forms the following operator formulation
\begin{equation*}
g_\mathrm{W}(\pi)^{-1}\circ g_F(\pi)\circ (-L^*_{\pi,\gamma})\succeq \lambda g_\mathrm{W}(\pi)^{-1}.
\end{equation*}
Thus 
\begin{equation*}
-L^*_{\pi,\gamma}\succeq \lambda g_{\mathrm{F}}(\pi)^{-1}.
\end{equation*}
This means that for any $h\in C^{\infty}(\Omega)$, we have
\begin{equation*}
\int_\Omega (h, -L^*_{\pi,\gamma} h)dx\geq \lambda \int_\Omega (h, g_{\mathrm{F}}(\pi)^{-1} h) dx,
\end{equation*}
i.e. 
\begin{equation*}
\int_\Omega \Big(\|\nabla h(x)\|^2-h(x)(\nabla h(x), \gamma)\Big) \pi(x)dx \geq \lambda \int_\Omega \|h(x)-\int_\Omega h(y)\pi(y)dy\|^2 \pi(x) dx,
\end{equation*}
which finishes the proof.
\end{proof}
\begin{remark}
In the proof, we use the following linear algebra. Consider any  symmetric matrices $A$, $B\in\mathbb{R}^{d\times d}$, and a non-symmetric matrix $C\in\mathbb{R}^d$. Then
\begin{equation*}
ABC\succeq \lambda A\quad \Rightarrow \quad C\succeq \lambda B^{-1}.     
\end{equation*}
We extend this fact into the comparison of operators in $L^2(\Omega)$, where $A=g_{\mathbb{W}}(p)^{-1}$, $B=g_{\mathbb{F}}(p)$ and $C=-L_{\pi, \gamma}^*$.
\end{remark}

\section{Examples}\label{sec4}
In this section, we provide an example for non-revisable stochastic dynamics in two dimensional spatial domain. In this case, we obtain the explicit convergence rate. 

Consider a non-reversible overdamped Langevin dynamics \cite{non-reversible} by
\begin{equation}\label{IJ}
dX_t=-(\mathbb{I}+\mathbb{J})\nabla U(X_t)dt+\sqrt{2}dB_t,
\end{equation}
where $\mathbb{I}\in \mathbb{R}^{2\times 2}$ is an identity matrix and $\mathbb{J}=\begin{pmatrix}
0& c\\
-c& 0
\end{pmatrix}$ is a skew-symmetric matrix with $c\in\mathbb{R}$. In this case, denote $x=(x_1, x_2)$. Then the invariant distribution satisfies
\begin{equation*}
\pi(x)=\frac{1}{Z}e^{-U(x)}, 
\end{equation*}
where $Z$ is a normalization constant. Here $b=-(\mathbb{I}+\mathbb{J})\nabla U$ and $\gamma=\nabla\log \pi-b=\mathbb{J}\nabla U$. 
We next formulate the explicit convergence rate for SDE \eqref{IJ} as follows. 
\begin{proposition}
Denote 
\begin{equation*}
\begin{split}
&\mathfrak{R}(x)\\
=&\begin{pmatrix}
\partial_{x_1x_1}U-\frac{c^2}{8}(\partial_{x_1}U)^2-\frac{c^2}{4}(\partial_{x_2}U)^2-c\partial_{x_1}U\partial_{x_2}U& \partial_{x_1x_2}U+\frac{c^2}{8}{\partial_{x_1}U\partial_{x_2}U}-\frac{c}{2}\Big((\partial_{x_2}U)^2-(\partial_{x_1}U)^2\Big)\\
\partial_{x_1x_2}U+\frac{c^2}{8}\partial_{x_1}U\partial_{x_2}U-\frac{c}{2}\Big((\partial_{x_2}U)^2-(\partial_{x_1}U)^2\Big) & \partial_{x_2x_2}U-\frac{c^2}{4}(\partial_{x_1}U)^2-\frac{c^2}{8}(\partial_{x_2}U)^2+c\partial_{x_1}U\partial_{x_2}U
\end{pmatrix}.
\end{split}
\end{equation*}
Then the convergence rate $\lambda$ is the smallest eigenvalue of $\mathfrak{R}(x)$ for all $x\in \Omega$.  
\end{proposition}
\begin{proof}
Denote $\gamma=\begin{pmatrix}
\gamma_1\\
\gamma_2
\end{pmatrix}$. We reformulate Lemma \ref{lem2} as follows. 
\begin{equation*}
\begin{split}
&\Gamma_2(f,f)+\Gamma_{\mathcal{I}}(f,f)\\
=&\|\nabla^2 f\|_{\mathrm{F}}^2+\nabla^2U(\nabla f, \nabla f)-\nabla^2f(\gamma, \nabla f)+\Delta f(\nabla f, \gamma)-(\nabla f, \nabla U)(\nabla f, \gamma)\\
=&\quad |\partial_{11}f|^2+2|\partial_{12}f|^2+|\partial_{22}f|^2+\nabla^2U(\nabla f, \nabla f)\\
&-\Big\{\partial_{11}f \partial_1f\gamma_1+\partial_{12}f(\gamma_1 \partial_2f+\gamma_2\partial_1f)+\partial_{22}f\gamma_2\partial_2f\Big\}\\
&+(\partial_{11}f+\partial_{22}f)(\partial_1f\gamma_1+\partial_2f\gamma_2)-(\partial_1f\partial_1U+\partial_2f\partial_2U)(\partial_1f\gamma_1+\partial_2f\gamma_2)\\
=&\quad(\partial_{11}f-\frac{1}{2}\partial_2f\gamma_2)^2+2(\partial_{12}f-\frac{\gamma_2\partial_1f+\gamma_1\partial_2f}{4})^2+(\partial_{22}f-\frac{1}{2}\partial_1f\gamma_1)^2\\
&+\Big(\partial_{11}U-\frac{\gamma_2^2}{8}-\frac{\gamma_1^2}{4}-\partial_1U\gamma_1\Big)(\partial_1f)^2+2\Big(\partial_{12}U-\frac{1}{2}(\partial_1U\gamma_2+\partial_2U\gamma_1)-\frac{1}{8}\gamma_1\gamma_2\Big)\partial_1f\partial_2f\\
&+\Big(\partial_{22}U-\frac{\gamma_2^2}{4}-\frac{\gamma_1^2}{8}-\partial_2U\gamma_2\Big)(\partial_2f)^2\\
=&\|\mathfrak{Hess}f\|_{\textrm{F}}^2+\mathfrak{R}(\nabla f,\nabla f).
\end{split}
\end{equation*}
In other words, we have
\begin{equation*}
\mathfrak{R}(x)=
\begin{pmatrix}
\partial_{11}U-\frac{\gamma_2^2}{8}-\frac{\gamma_1^2}{4}-\partial_{1}U \gamma_1& \partial_{12}U-\frac{\gamma_1 \gamma_2}{8}-\frac{\partial_{1}U \gamma_2+\partial_{2}U \gamma_1}{2}\\
 \partial_{12}U-\frac{\gamma_1 \gamma_2}{8}-\frac{\partial_{1}U\gamma_2+\partial_{2}U \gamma_1}{2} & \partial_{22}U-\frac{\gamma_2^2}{4}-\frac{\gamma_1^2}{8}-\partial_{2}U \gamma_2
\end{pmatrix}.
\end{equation*}
Using the fact $\gamma=\nabla\log \pi-b=\mathbb{J}\nabla U$, we prove the result.
\end{proof}

\begin{example}
Consider 
\begin{equation*}
U(x_1,x_2)=\frac{x_1^2+x_2^2}{2}.
\end{equation*}
Then 
\begin{equation*}
\mathfrak{R}(x)
=\begin{pmatrix}
1-\frac{c^2}{8}x_1^2-\frac{c^2}{4}x_2^2-cx_1x_2& \frac{c^2}{8}x_1x_2+\frac{c}{2}(x_1^2-x_2^2)\\
\frac{c^2}{8}x_1x_2+\frac{c}{2}(x_1^2-x_2^2)& 1-\frac{c^2}{4}x_1^2-\frac{c^2}{8}x_2^2+cx_1x_2
\end{pmatrix}.
\end{equation*}
Here we plot the smallest eigenvalue of $\mathfrak{R}(x)$ numerically, for different choices of $U$ for a given constant $c$.   
 \begin{figure}[H]
    \includegraphics[scale=0.25]{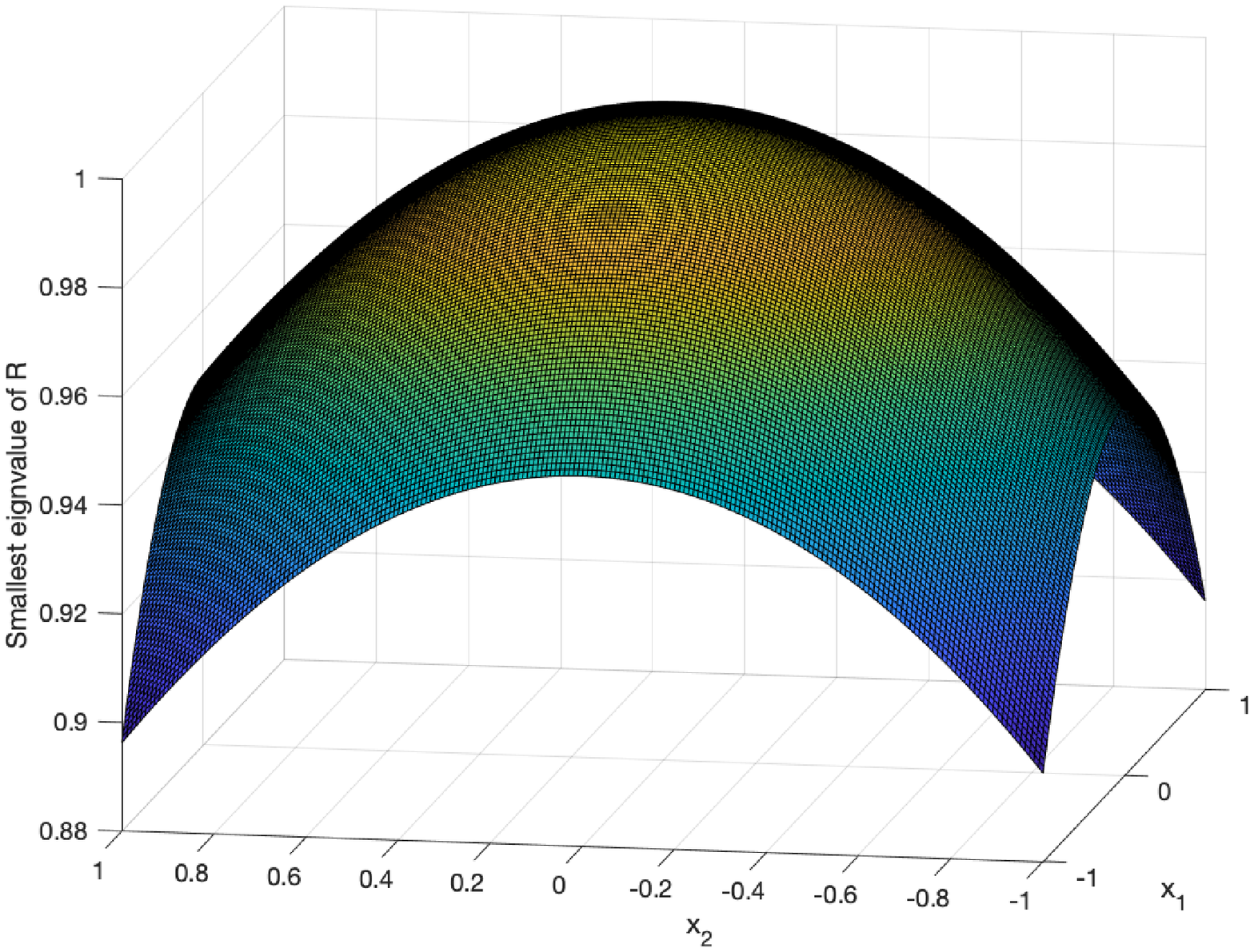}\includegraphics[scale=0.25]{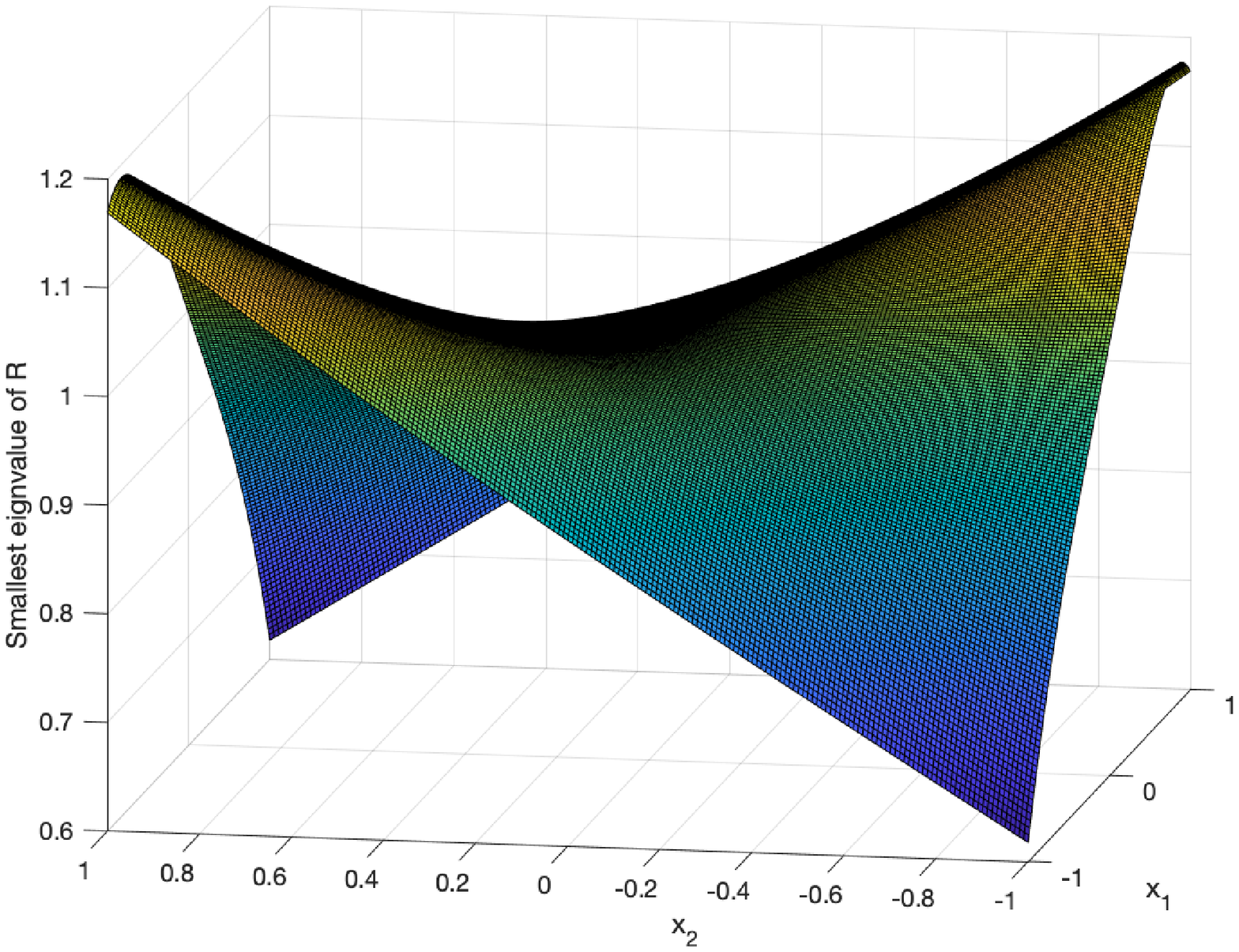}
    \caption{Illustration of the convergence rate depending on $c=0.1$ on a domain $[-1,1]^2$. Left is for $U(x)=\frac{x_1^2+x_2^2}{2}$. Right is for $U(x)=\frac{x_1^2+3x_2^2}{2}$.
    }
    \label{figure}
\end{figure}
From Figure \ref{figure}, we know that if $\lambda_{\min}(\mathfrak{R}(x))\geq \lambda>0$ on a region $\Omega$, then the relative Fisher information decays exponentially fast for \eqref{FPE} defined on $\Omega$. 
If $U(x)=\frac{x_1^2+x_2^2}{2}$, then adding the non-gradient drift does not improve the convergence rate, which is  $\lambda_{\min}(\nabla^2U)=1$. If $U(x)=\frac{x_1^2+3x_2^2}{2}$, then there exists a region, explicitly shown in the Figure \ref{figure}, where the related convergence rate and log-Sobolev inequality hold with a constant larger than $\lambda_{\min}(\nabla^2U)=1$. 
\end{example}

\section*{Appendix: Geometric calculations in $L^2$--Wasserstein space}
In appendix, we formulate the derivation of information Gamma calculus based on geometric calculations in $L^2$--Wasserstein space. 

\subsection{Non-reversible SDEs via perturbed Wasserstein gradient flows}
We firstly reformulate non-reversible SDE in $L^2$--Wasserstein space. This explains the derivation of this paper. Consider the non-reversible SDE \eqref{a} by
\begin{equation*}
d X=(\nabla \log\pi-\gamma)dt+\sqrt{2}dB_t. 
\end{equation*}
Denote $X\sim p$, where $p(t,x)$ is the probability density function of stochastic process \eqref{a}. We observe that 
\begin{equation*}
d\tilde X=\Big(\nabla\log\pi-\gamma-\nabla\log p\Big)dt,
\end{equation*}
provides the same probability transition equation for the Fokker-Planck equation of \eqref{a}. A further reformulation of SDE \eqref{a} is as follows. Define $\phi$ as a vector field satisfying the elliptical equation 
\begin{equation*}
\nabla\cdot(p \nabla\phi)=\nabla\cdot(p\gamma). 
\end{equation*}
Consider
\begin{equation*}\label{reformula}
d\tilde{\tilde X}=\Big(-\nabla\log\frac{p}{\pi}+\nabla\phi\Big)dt. 
\end{equation*}
We notice that the Kolmogorov forward equations for processes $X$, $\tilde X$ and $\tilde{\tilde X}$ satisfy the same Fokker-Planck equation:
\begin{equation*}
\begin{split}
\partial_tp=&\Delta p-\nabla\cdot(p\nabla\log \pi)+\nabla\cdot(p\gamma)\\
=&\nabla\cdot(p\nabla\log p)-\nabla\cdot(p\nabla\log \pi)+\nabla\cdot(p\gamma)\\
=&\nabla\cdot(p\nabla(\log\frac{p}{\pi}+\phi)),
\end{split}
\end{equation*}
where we apply the fact that $p\nabla\log p=\nabla p$ in the second equality. 
\subsection{Geometric calculations in $L^2$--Wasserstein space}
We next explain information Gamma calculus by using perturbed gradient flow in $L^2$--Wasserstein space. 

We first recall the metric and the gradient operator in $L^2$-Wasserstein space \cite{OV, Villani2009_optimal}. Denote the space of smooth positive probability densities by 
\begin{equation*}
\mathcal{P}=\Big\{p(x)\in C(\Omega)\colon \int_\Omega p(x)dx=1,\quad p(x)>0\Big\}.
\end{equation*}
Given a density $p\in \mathcal{P}$, the tangent space of $\mathcal{P}$ forms 
\begin{equation*}
T_p\mathcal{P}=\Big\{\sigma(x)\in C(\Omega)\colon \int_\Omega \sigma(x)dx=0\Big\}.
\end{equation*}
Given $p\in\mathcal{P}$, define the elliptical operator by
\begin{equation*}
\Delta_p=\nabla_x\cdot(p(x)\nabla_x)~\colon~C^{\infty}(\Omega)\rightarrow C^{\infty}(\Omega). 
\end{equation*}
\begin{definition}[$L^2$--Wasserstein metric]
The $L^2$--Wasserstein metric refers to the following bilinear norm $g_{\mathrm{W}}\colon \mathcal{P}\times T_p\mathcal{P}\times T_p\mathcal{P}\rightarrow\mathbb{R}$, such that for any $\sigma_i\in T_p\mathcal{P}$, $i=1,2$,
\begin{equation}\label{metric}
g_{\mathrm{W}}(p)(\sigma_1, \sigma_2)=\int_\Omega (\sigma_1, (-\Delta_p)^{-1}\sigma_2)dx.
\end{equation}
\end{definition}
A known fact is that the gradient operator of KL divergence is given by 
\begin{equation*}
\begin{split}
\textrm{grad}_{\mathrm{W}}\mathrm{D}_{\mathrm{KL}}=&\Big((-\Delta_p)^{-1}\Big)^{-1}\delta_p\mathrm{D}_{\mathrm{KL}}
=-\Delta p+\nabla\cdot(p\nabla\log \pi).
\end{split}
\end{equation*}
where $\delta_p$ is the $L^2$ first variation operator w.r.t. $p$ and $\delta_p\mathrm{D}_{\mathrm{KL}}=\log\frac{p}{\pi}+1$. This operator corresponds to the Fokker-Planck equation of reversible SDEs. In this notation, 
the Fisher information forms 
\begin{equation*}
\begin{split}
\mathcal{I}(p\|\pi)=\int_\Omega \|\nabla\log\frac{p}{\pi}\|^2 p dx=g_{\mathrm{W}}(p)(\textrm{grad}_{\mathrm{W}}\mathrm{D}_{\textrm{KL}}, \textrm{grad}_{\mathrm{W}}\mathrm{D}_{\textrm{KL}}). 
\end{split}
\end{equation*}

We next reformulate the Fokker-Planck equation for non-reversible SDE \eqref{a} by
\begin{equation}\label{PFPE}
\begin{split}
\partial_t p
=&-\textrm{grad}_{\mathrm{W}}\mathrm{D}_{\mathrm{KL}}+\nabla\cdot(p\gamma),
\end{split}
\end{equation}
Hence we can view \eqref{PFPE} as the perturbed gradient flow. And our derivation uses the following geometric calculations in $L^2$--Wasserstein space. 
\begin{proposition}
Denote  $\phi$ as a vector field satisfying the elliptical equation 
\begin{equation*}
\nabla\cdot(p \nabla\phi)=\nabla\cdot(p\gamma)=p(\nabla f, \gamma),
\end{equation*}
where $f=\log\frac{p}{\pi}$. Then 
\begin{equation*}
\begin{split}
\mathrm{Hess}_{\mathrm{W}}\mathrm{D}_{\mathrm{KL}}(f,\phi)=&\int_\Omega \Gamma_{\mathcal{I}}(f,f)p(x)dx,
\end{split}
\end{equation*}
where $\mathrm{Hess}_{\mathrm{W}}$ is the Hessian operator defined in the cotangent bundle of $L^2$--Wasserstein space; see details in \cite[Proposition 20]{Li2018_geometrya}. 
\end{proposition}
\begin{proof}
The proof follows a direct calculation. Notice that Fokker-Planck equation \eqref{FPE} forms \eqref{PFPE}, i.e. 
\begin{equation*}
\partial_t p=-(-\Delta_p)(\log\frac{p}{\pi}+\phi)=\nabla\cdot\Big(p\nabla(f+\phi)\Big).
\end{equation*}
Along the Fokker-Planck equation \eqref{FPE}, we have
\begin{equation*}
\begin{split}
\frac{d}{dt}\mathcal{I}(p\|\pi)=&\frac{d}{dt}g_{\mathrm{W}}(\textrm{grad}_{\mathrm{W}}\mathrm{D}_{\mathrm{KL}}, \textrm{grad}_{\mathrm{W}}\mathrm{D}_{\mathrm{KL}})\\
=&2\textrm{Hess}_{\mathrm{W}}\mathrm{D}_{\mathrm{KL}}(-(f+\phi), f)\\
=&-2\Big(\textrm{Hess}_{\mathrm{W}}\mathrm{D}_{\mathrm{KL}}(f,f)+\textrm{Hess}_{\mathrm{W}}\mathrm{D}_{\mathrm{KL}}(f,\phi)\Big).
\end{split}
\end{equation*}
From the Hessian operator in $L^2$--Wasserstein space \cite{Li2018_geometrya, LiG2, Villani2009_optimal}, we have
\begin{equation*}
\textrm{Hess}_{\mathrm{W}}\mathrm{D}_{\mathrm{KL}}(f,f)=\int_\Omega \Gamma_2(f,f)p dx. 
\end{equation*}
And the information Gamma operator is the remaining term, i.e. 
\begin{equation*}
\textrm{Hess}_{\mathrm{W}}\mathrm{D}_{\mathrm{KL}}(f,\phi)=\int_\Omega \Gamma_{\mathcal{I}}(f,f)p dx. 
\end{equation*}
\end{proof}
Finally, we summarize that the information Gamma calculus satisfies  
\begin{equation*}
\begin{split}
\textrm{Hess}_{\mathrm{W}}\mathrm{D}_{\mathrm{KL}}(f,\phi+f)=&\int_\Omega \Big(\Gamma_2(f,f)+\Gamma_{\mathcal{I}}(f,f)\Big)p dx\\
=&\int_\Omega \Big(\|\mathfrak{Hess}f\|^2_{\mathrm{F}}+\mathfrak{R}(\nabla f, \nabla f)\Big) p dx\\
=&\int_\Omega \Big(\|\nabla^2 f\|^2+\mathfrak{R}_{\mathrm{AC}}(\nabla f, \nabla f)\Big)p dx.
\end{split}
\end{equation*}
It is a Hessian operator of KL divergence for asymmetric vectors in $L^2$--Wasserstein space. Due to the asymmetry nature, there exists multiple formulations of tensors.  
\end{document}